\documentclass[11 pt]{amsart}
\usepackage{amsmath}
\usepackage{amssymb}
\usepackage{amsfonts}
    
\newtheorem{theorem}{Theorem}[section]
\newtheorem{lemma}[theorem]{Lemma}
\newtheorem{corollary}[theorem]{Corollary}
\newtheorem{proposition}[theorem]{Proposition}

\newtheorem{definition}[theorem]{Definition}
\newtheorem{example}[theorem]{Example}

\theoremstyle{remark}
\newtheorem{remark}[theorem]{\bf Remark}

\theoremstyle{remarks}

\numberwithin{equation}{section}
  \usepackage{a4wide}

\def \d{{\delta}}
\def \e{{\varepsilon}}
\def \g{{\gamma}}
\def \G{{\Gamma}}

\def \l{{\lambda}}

\def \p{{\varphi}}
\def \m{{\mu}}
\def \s{{\sigma}}
\def \dd{{\rm d}}
\def \noi{{\noindent}}
  \def \qq{{\qquad}}

\def\E{{\mathbb E \,}}

\def \T{{\mathbb T}}

\def\P{{\mathbb P}}

\def\R{{\mathbb R}}

\def\Z{{\mathbb Z}}

\def\N{{\mathbb N}}

\def\C{{\mathbb C}}

 \scrollmode
 \font\sevenrm= cmr10 at 7,3 pt

\def\ddate {\sevenrm \ifcase\month\or January\or
February\or March\or April\or May\or June\or July\or
August\or September\or October\or November\or December\fi\! {\the\day}, \!{\sevenrm\the\year}}

 \title[An application of  Brascamp-Lieb's inequality]{An application of  Brascamp-Lieb's inequality}
 
\begin{document}   
 
\author{Michel J.\,G. Weber}

\address{IRMA, 10 rue du G\'en\'eral Zimmer,
67084 Strasbourg Cedex, France}
\email{michel.weber@math.unistra.fr}


\maketitle 
 \begin{abstract}  We use Brascamp-Lieb's inequality to obtain new decoupling inequalities for general  Gaussian vectors, and for stationary cyclic Gaussian processes. In the second case, we use a version by Bump and Diaconis of the strong Szego limit theorem. This extends results of Klein, Landau and Shucker.
 \end{abstract}

\section{\bf Introduction-Results.}\label{s1}
Let   $X=\{X_j, j\in \Z\}$ be a centered Gaussian stationary  sequence, and let $\g(n) = \E X_0X_n$, $n \in \Z$.   We assume that  $X$ is strongly mixing or   equivalently, that $\lim_{n\to \infty}  \g(n)  =0$. When $ \g(n)$ tends sufficiently quickly to $0$, more independence is naturally gained in the structure of $X$.  This can be quantified under the form of  a decoupling inequality. For instance, if 
\begin{eqnarray}\label{hyp.dec} p(X) \ =\ \sum_{n\ge 0} {|\g(n) |\over \g(0)}<\infty ,
\end{eqnarray}
then for any finite collection $\{f_j, j\in J\}$ of  complex-valued
Borel-measurable functions,
\begin{eqnarray}\label{dec}
\Big|\E  \prod_{j\in J}f_j\big(X_j\big) \Big|\ \le\  \prod_{j\in J}\big\|f_j\big(X_0\big) \big \|_{p(X)}.\end{eqnarray}
  This  remarkable  inequality, which 
  so nicely condenses the independence properties of these Gaussian sequences,  is Theorem 3 ($d=1$) in Klein, Landau and Shucker \cite{KLS}. 
    \vskip 2 pt
  Clearly $|\E  \prod_{j\in J}f_j(X_j)|$ measures the degree of independence between the  random variables $f_j(X_j)$. 
 An immediate consequence of \eqref{dec} and of well-known Kathri-Sid\'ak's inequality, is that under assumption \eqref{hyp.dec}, we have  the following sharp two-sided estimate,
\begin{equation}\label{ks} \prod_{j\in J} \P \{  |X_j|\le x \}\le \P\big\{ \sup_{j\in J}  |X_j|\le x\big\}\le \prod_{j\in J} \P \{  |X_j|\le x \}^{1/p(X)},
\end{equation}
where $J$ is any finite index  and  $x$ any non-negative real.  \vskip 3 pt 
     However, one is often faced with probabilistic questions where  $p(X)=\infty$,
  or simply,  the process $X$ is not stationary in the sense required in \cite{KLS}.   At our knowledge, no extension of \eqref{dec} beyond condition \eqref{hyp.dec} exists in the literature, and it is naturally interesting to search what form could take a decoupling inequality  when \eqref{hyp.dec} fails to be satisfied. 
   \vskip 2 pt
   This is the question we address and study 
  in this work, which is also  somehow developing 
   the recent paper \cite{W}.
 We 
  clarify at this stage that our goal is to obtain results   valid for a broad range of Gaussian processes, and thus  not  (possibly)  quite sharp  estimates concerning specific cases, which is another problematic. Our aim is  also to link the question considered with the general theory of Toeplitz forms, and  draw the attention of the reader to the 
 interest of this  connection.
This is in that sense continuing the study made in  section 5 of \cite{W}. 
  \vskip 2 pt 
The role of the stationarity assumption of $X$ in  \cite{KLS}    is   crucial. The proof of \eqref{dec}  much relies  on an analytic inequality due to Brascamp and Lieb, which is of  relevance in the present work. 
  \vskip 2 pt 
    
A first  natural question can be stated as follows.   
  What form can take the decoupling inequality \eqref{dec} for an arbitrary Gaussian vector? As the law of a Gaussian vector, or more generally of a Gaussian process is completely characterized by its covariance function,   one can make  the question more consistent by asking which characteristics of the covariance matrix $\{\E X_iX_j\}_{i,j=1}^n$ of $X$ should be involved (and are to be evaluated):
  a particular function of its eigenvalues, or simply its determinant? It turns out that only the determinant suffices.  More precisely, we  prove  a general decoupling inequality, 
free of stationarity assumption.   \vskip 2 pt 
Before stating it, we first  extend the notion of decoupling coefficient introduced in  \cite{KLS} to arbitrary Gaussian vectors. 
 \begin{definition} Let $ X=\{X_i, 1\le i\le n\}$ be a centered Gaussian vector with non-degenerated components. The decoupling coefficient $p(X)$ of $X$ is defined by
 \begin{eqnarray*}
 p(X)&=& \max_{i=1}^n\sum_{1\le j\le n} \frac{|\E X_iX_j|}{\E X_i^2}
\,.
\end{eqnarray*}
\end{definition}
This is a natural characteristic of $X$. When $X$ is stationary, 
\begin{eqnarray*}
p(X)&=&\max_{i=1}^n\sum_{1\le j\le n} \frac{|\g(i-j)|}{\g(0)},\end{eqnarray*}
and so 
$$\sum_{1\le h\le n-1} \frac{|\g(h)|}{\g(0)}\le p(X)\le 2\sum_{1\le h\le n-1} \frac{|\g(h)|}{\g(0)}.$$ Further $p(X)=1$ if and only if $X$ has independent components. Some classes of examples with $p(X)\ll n$ or $p(X)\asymp n$ are given in section \ref{s4}.
\vskip 3 pt  
   
Our first main result states as follows.\vskip 2 pt  
\begin{theorem}\label{dec.ineq} Let $ X=\{X_i, 1\le i\le n\}$ be a centered Gaussian vector such that $\E X_i^2=\s_i^2>0$ for each  $1\le i\le n$, and with positive definite  covariance matrix $C$.
 Let $p$ be such that
\begin{eqnarray}\label{cond}
p\, \ge \, 2\, p(X).
\end{eqnarray}
 Then for any complex-valued measurable  functions $f_1, \ldots, f_n$ such that $f_i\in L^p(\R)$, for all $1\le i\le n$, the following inequality holds true, 
\begin{eqnarray}\label{ineq}
\bigg|\,\E\Big( \prod_{i=1}^n f_i(X_i)\Big)\,\bigg|
&\le &\frac{ 2^{\frac{n}{2}(1-\frac{1}{p})}\big(\prod_{i=1}^n\s_i\big)^{\frac{1}{p}}}
 {\det(C)^{\frac{1}{2p}}}\ \prod_{i=1}^n\Big( \E   f_i(X_i)^p\Big)^\frac{1}{p}.
\end{eqnarray}
 \end{theorem}

\vskip 6 pt 
From Theorem \ref{dec.ineq} and Kathri-Sid\'ak's inequality we also get, 
 \begin{corollary}\label{co1}Let $ X=\{X_i, 1\le i\le n\}$ be a centered Gaussian vector such that $\E X_i^2=\s_i>0$, $1\le i\le n$, and  with positive definite  covariance matrix $C$. Assume that assumption \eqref{cond} is fulfilled for some $p\ge 2$. Then  for any $\e_i >0$, $i=1,\ldots, n$,
 \begin{eqnarray*}
\prod_{i=1}^n\P\big\{  |X_i|\le \e_i\big\} \ \le \ \P\Big\{ \sup_{i=1}^n \frac{|X_i|}{\e_i}\le 1\Big\}
&\le &
\frac{2^{\frac{n}{2}}}{\det(C)^{\frac{1}{2p}} }  \ 
\prod_{i=1}^n\Big( \frac{\,\s_i}{\sqrt 2} \ \P\big\{  |X_i|\le \e_i\big\}\Big)^\frac{1}{p}.
\end{eqnarray*}
\end{corollary}


\vskip 6 pt For estimating $\det(C)$, we  place ourselves in the setting of Toeplitz matrices theory where this important question has been and is still  much investigated. A salient aspect of  this theory is that $\det(C)$ can be computed, sometimes  with high degree of accuracy. We refer to the nice book of  Grenander and 
Szeg\"o \cite{GS}, we also refer to \cite{W} for a general presentation of the methods used, except for the Laplace transform method,
essentially  in the setting of stationary Gaussian processes.
Let $f(t)= \sum_{-\infty}^\infty d_n e^{int}$ be a function on the unit circle $\T$. 
Let $T_{n-1}(f)$ be the Toeplitz matrix defined by $T_{n-1}(f)= \{ d_{j-i}\}_{i,j=0}^{n-1}$ and let $D_{n-1}(f)=\det(T_{n-1}(f))$.

\vskip 2 pt
This   corresponds to the case when   $X$ has a 
spectral density function $f(t)$, summable over $[-\pi,\pi]$, and is thus of relevance in our setting.
 Indeed, as
\begin{eqnarray}\label{te} d_{n}=\frac1{2\pi}\int_{- \pi}^{ \pi} e^{-in t} f(t)   \dd t, \qq  \quad  n\in \Z ,
\end{eqnarray}
$T_{n-1}(f)$ is just the $n$-th finite section
of the infinite Toeplitz matrix given by the covariance matrix of the process $X$.  Further as $f\in L^1([-\pi, \pi])$,  by the Riemann-Lebesgue lemma, we have
 $ \lim_{n\to \infty}  d_n =0$.  
 \vskip 3 pt

  In the considerable literature  on Toeplitz operators and determinants, $f$ is usually called a symbol or a generating function (generating $(T_{n-1}(f))_n$) and $f$ needs not being a density function. Toeplitz determinants with rational symbols  occur for instance in statistical mechanics and quantum mechanics, see \cite{BF}. They can be calculated using a formula obtained by Day \cite{D}.
\vskip 2 pt
For Toeplitz matrices generated by a density function, $D_{n-1}(f)$ can also be expressed as an integral over the unitary group $U(n)$, by means of the Heine-Szeg\"o identity, 
\begin{equation}\label{hs}D_{n-1}(f) \, =\, \int_{U(n)} \Phi_{n,f}(g) \dd g.
\end{equation}
Here 
the integration path is taken with respect to the normalized Haar measure on $U(n)$, and $\Phi_{n,f} (g)$ is defined by $\Phi_{n,f} (g)= f(t_1)\ldots f(t_n)$, where $t_1,\ldots,t_n$ are  the eigenvalues of $g$. 
This  identity is the starting point of  the proof of a nice form of  the strong Szeg\"o limit theorem established in   \cite{BD} by Bump and Diaconis. 
 
 \vskip 3 pt Using their result we also prove

\begin{theorem}\label{p2} Let   $X=\{X_j, j\in \Z\}$ be a centered Gaussian stationary  sequence with unit variance and 
spectral density function
$f(t)$.   Let  $\log f(t) = \sum_{k\in \Z} c_ke^{ikt}$ where the $c_k$ satisfy the following conditions
\begin{equation}\label{c1} \sum_{k\in \Z} |c_k|<\infty,
\end{equation}
\begin{equation}\label{c2} \sum_{k\in \Z} |k|\,|c_k|^2<\infty.
\end{equation}

  Then there exist  reals $\d_n\downarrow 0$, such that for any integer $n\ge 2$, 
 any complex-valued measurable  functions $f_1, \ldots, f_n$ with  $f_i\in L^p(\R)$, for all $1\le i\le n$,  
 where 
\begin{eqnarray}\label{cond1a}
p\, \ge \, 2\,p(X),
\end{eqnarray}
the following inequality holds true, 
\begin{eqnarray*}
\bigg|\E\Big( \prod_{i=1}^n f_i(X_i)\Big)\bigg|
&\le & 
 \frac{(1+\d_n)\, 2^{\frac{n}{2}} }{(2b(f)G(f))^{\frac{n }{2p}}}\  \prod_{i=1}^n\Big( \E |  f_i(X_i)|^p\Big)^\frac{1}{p},
\end{eqnarray*}
where $b(f)=\exp\big\{   \sum_{k=1}^\infty kc_kc_{-k}\big\}$
 and $G(f)$ is the geometric mean of $f$, namely
\begin{equation}\label{G(f)}
G(f) =\exp\big\{\frac{1}{2\pi}\int_{- \pi}^{ \pi}
\log f(t) \dd t \big\}.
\end{equation} 

 \end{theorem}
\vskip 2 pt 
  
\begin{corollary}\label{co2}
Let   $X=\{X_j, j\in \Z\}$ be a centered Gaussian stationary  sequence with unit variance and 
spectral density function
$f(t)$ satisfying  conditions \eqref{c1} and \eqref{c2}.
 \vskip 2 pt  Then there exist  reals $\d_n\downarrow 0$, such that for any integer $n\ge 2$, 
 any  $p$ satisfying \eqref{cond1a},
we have for any $\e_i >0$, $i=1,\ldots, n$,
 \begin{eqnarray*}
\P\Big\{ \sup_{i=1}^n |X_i|\le \e_i\Big\}
&\le &
 \frac{(1+\d_n)\, 2^{\frac{n}{2}} }{(2b(f)G(f))^{\frac{n }{2p}}}\ \prod_{i=1}^n\big( \P\big\{  |X_0|\le \e_i\big\}\big)^\frac{1}{p}
.
\end{eqnarray*}
\end{corollary}


\section{\bf Proof of Theorem \ref{dec.ineq}.}\label{s2}
We first state the proposition below which follows from Theorem 6 in Brascamp and Lieb \cite{BL}. We also refer to \cite{KLS}. It should be indicated here that all that is required for the application of this Theorem, is that the   matrix be positive definite. This one  is written  in terms of its eigenvectors and eigenvalues, and the eigenvectors are  $a^j$, $j=k+1,\ldots, k+m$, which have nothing to do with the vectors 
 $a^j$, $j=1,2,\ldots, k$ of their Theorem 1. This point was clarified to the author by Abel Klein \cite{K}. Introduce some notation. Let    $I$  be the $n\times n$ identity matrix and let $\underline b=(b_1, \ldots ,b_n)\in \R^n$. Then $  I(\underline b)$ will denote throughout the diagonal matrix whose values on the diagonal are the corresponding values of $\underline b$. Also, when $b_i\neq 0$ for each $i=1, \ldots, n$, we will  use the notation $\underline b^{-1}=(b_1^{-1}, \ldots ,b_n^{-1})$.
\begin{proposition} \label{p1}Let $1\le p<\infty$. Let $B$ be a positive definite $n\times n$ matrix. Then for any measurable  functions $g_1, \ldots, g_n$ such that $g_i\ge 0$ and $g_i\in L^p(\R)$, $1\le i\le n$, the following inequality holds true, 
\begin{eqnarray}\label{BL.ineq}
\int_{\R^n}\Big( \prod_{i=1}^n g_i(x_i)\Big) \exp\Big\{ -\frac12\langle \underline x, B\underline x\rangle\Big\}\dd \underline x&\le & E_B\ \prod_{i=1}^n\Big( \int_{\R}    g_i(x)^p\dd x\Big)^\frac{1}{p},
\end{eqnarray}
where
\begin{eqnarray}\label{EB}
E_B\,=\,(2\pi)^{\frac{n}{2}(1-1/p)}p^{\frac{n}{2p}} \, \sup_{b_i>0\atop i=1,\ldots,n}\frac{\prod_{i=1}^nb_i^\frac{1}{2p}}{ \det (B+I(\underline b))^\frac12}.
\end{eqnarray}
\end{proposition}
 \vskip 2 pt 
\begin{remark}\label{REB}
 The constant $E_B$  is defined in Theorem 6 by 
\begin{eqnarray*}
E_B&=&\sup_{b_i>0\atop i=1,\ldots,n}\frac{\int_{\R^n}\big( \prod_{i=1}^n \exp\{-\frac12 b_ix_i^2\}\big) \exp\{ -\frac12\langle \underline x, B\underline x\rangle\}\dd \underline x}{\prod_{i=1}^n\big( \int_{\R}   \exp\{-\frac{p}{2} b_ix^2\}\dd x\big)^\frac{1}{p}},
\end{eqnarray*}
namely  inequality \eqref{BL.ineq}  is maximal when the $g_i$'s are Gaussian. 
\vskip 2 pt 
But it is elementary
that (see also \cite{KLS}, p.\,705, after (4))

 \begin{eqnarray*}
\frac{\int_{\R^n}\big( \prod_{i=1}^n \exp\{-\frac12 b_ix_i^2\}\big) \exp\{ -\frac12\langle \underline x, B\underline x\rangle\}\dd \underline x}{\prod_{i=1}^n\big( \int_{\R}   \exp\{-\frac{p}{2} b_ix^2\}\dd x\big)^\frac{1}{p}}
 &=& \frac{\int_{\R^n} \exp\{ -\frac12\langle \underline x, (B+I(\underline b))\underline x\rangle\}\dd \underline x}{\prod_{i=1}^n\big( \frac{2\pi}{pb_i}  \big)^\frac{1}{2p}}
\cr &=&(2\pi)^{\frac{n}{2}(1-\frac{1}{p})}p^{\frac{n}{2p}} \,\frac{\prod_{i=1}^nb_i^\frac{1}{2p}}{ \det (B+I(\underline b))^\frac12}.
\end{eqnarray*}
So that\begin{eqnarray*}
E_B\,=\,(2\pi)^{\frac{n}{2}(1-\frac{1}{p})}p^{\frac{n}{2p}} \, \sup_{b_i>0\atop i=1,\ldots,n}\frac{\prod_{i=1}^nb_i^\frac{1}{2p}}{ \det (B+I(\underline b))^\frac12}.
\end{eqnarray*}
\end{remark}
\vskip 3 pt
\begin{remark}\label{RKLS}
In the proof of Theorem 3 in \cite{KLS}, Klein, Landau and Shucker  apply Theorem 6 under the form of that Proposition, p.\,705, with the choice $B=C^{-1}-\frac{1}{pc}I$, where  $C$ is the covariance matrix of the process $X$, $p$ is the decoupling coefficient of $X$, $c=\E X_0^2$. Further $g_i(x)= f_i(x) e^{-(1/2pc)x^2}$, for $i=1,\ldots, n$. 

This requires that $B$ is positive definite, or equivalenly that $pcI-C$ is positive definite. This is ensured by the choice of $p$ made in \cite{KLS}.   
\end{remark}
\vskip 3 pt

\vskip 5 pt
In the next lemma, we establish a general bound of $E_B$.  
\begin{lemma}\label{l1}
 \begin{eqnarray*}
E_B&\le&
\frac{(2\pi)^{\frac{n}{2}(1-\frac{1}{p})}}{\det(B)^{\frac12(1-\frac{1}{p})}} .
\end{eqnarray*}

\end{lemma}

\begin{proof}  We use the following Lemma. 
\begin{lemma}[\cite{B}, Th. 4, p.\,128] If $U$ and $V$ are positive definite matrices, then
\begin{eqnarray*}
\det (\l U+(1-\l) V)\ge \det(U)^\l \det (V)^{1-\l},
\end{eqnarray*}
for any $0\le \l \le 1$. \end{lemma} 
Therefore 
\begin{eqnarray*}
\det \big( U+(\frac{1-\l}{\l}) V\big)\ge \l^{-n}\det(U)^\l \det (V)^{1-\l},
\end{eqnarray*}
if $0< \l \le 1$. We apply this with the choice $U= B$, $V=(\frac{\l}{1-\l} )I(\underline b)$.  We get
\begin{eqnarray*}
\det \big( B +I(\underline b)\big)&\ge& \l^{-n}\det(B)^\l \det \big((\frac{\l}{1-\l} )I(\underline b)\big)^{1-\l}
\cr &=& \l^{-n}\det(B)^\l\,\big(\frac{\l}{1-\l} \big)^{n(1-\l)}  \det (I(\underline b))^{1-\l}
\cr &=&\Big(\frac1\l\,\big(\frac{\l}{1-\l} \big)^{1-\l}\Big)^n\det(B)^\l\   \prod_{i=1}^nb_i^{1-\l}\, .
\end{eqnarray*}
Consequently, 
\begin{eqnarray*}
 \frac{\prod_{i=1}^nb_i^\frac{1}{2p}}{ \det (B+I(\underline b))^\frac12}
 &\le &\frac{1}{\det(B)^{\frac{\l}{2}}}\Big(\l\big(\frac{1-\l}{\l}\big)^{1-\l}\Big)^{\frac{n}{2}} \prod_{i=1}^nb_i^{\frac{1}{2}(\frac{1}{p}-(1-\l))}
 \cr  &= &\frac{1}{\det(B)^{\frac{\l}{2}}}\Big(\l^\l (1-\l)^{1-\l}\Big)^{\frac{n}{2}} \prod_{i=1}^nb_i^{\frac{1}{2}(\frac{1}{p}-(1-\l))}.
\end{eqnarray*}
Take $\l=1-\frac{1}{p}$ and note that  $(1-\l)^{1-\l}=p^{-\frac{1}{p}} $. We obtain
\begin{eqnarray*}
\frac{\prod_{i=1}^nb_i^\frac{1}{2p}}{ \det (B+I(\underline b))^\frac12}
 &\le &\frac{1}{\det(B)^{\frac12(1-\frac{1}{p})}}  p^{-\frac{n}{2p}}\Big(1-\frac{1}{p}\Big)^{\frac{n}{2}(1-\frac{1}{p})}
 .
\end{eqnarray*}
Therefore
\begin{eqnarray*}
 (2\pi)^{\frac{n}{2}(1-\frac{1}{p})}p^{\frac{n}{2p}}\frac{\prod_{i=1}^nb_i^\frac{1}{2p}}{ \det (B+I(\underline b))^\frac12}
 &\le &\frac{(2\pi)^{\frac{n}{2}(1-\frac{1}{p})}}{\det(B)^{\frac12(1-\frac{1}{p})}}p^{\frac{n}{2p}}\!\cdot\! p^{-\frac{n}{2p}}\Big(1-\frac{1}{p}\Big)^{\frac{n}{2}(1-\frac{1}{p})}
 \cr &= &\frac{(2\pi)^{\frac{n}{2}(1-\frac{1}{p})}}{\det(B)^{\frac12(1-\frac{1}{p})}}\Big(1-\frac{1}{p}\Big)^{\frac{n}{2}(1-\frac{1}{p})} .
\end{eqnarray*}

Whence 
\begin{eqnarray*}
E_B&\le&
\frac{(2\pi)^{\frac{n}{2}(1-\frac{1}{p})}}{\det(B)^{\frac12(1-\frac{1}{p})}} .
\end{eqnarray*}

\end{proof}

\begin{proof}[Proof of Theorem \ref{dec.ineq}] It  suffices to prove  inequality \eqref{ineq} when $f_i$ are real-valued and non-negative, for all $i=1,\ldots,n$. 
Let $\underline{\g}=(\s_1^{2}, \ldots, \s_n^{2})$. We apply Proposition \ref{p1}
 with $B=C^{-1}-\frac{1}{p}I(\underline{\g}^{-1})$,    $g_i(x)= |f_i(x)| e^{-x^2/(2p\s_i^2) }$,   $i=1,\ldots, n$.
\vskip 2 pt
We get by using also Lemma \ref{l1},
\begin{eqnarray*}
\int_{\R^n}\Big( \prod_{i=1}^n g_i(x_i)\Big) \exp\Big\{ -\frac12\langle \underline x, B\underline x\rangle\Big\}\dd \underline x
&\le & E_B\prod_{i=1}^n\Big( \int_{\R}    g_i(x)^p\dd x\Big)^\frac{1}{p}
\cr &\le & \frac{(2\pi)^{\frac{n}{2}(1-\frac{1}{p})}}{\det(B)^{\frac12(1-\frac{1}{p})}}\ \prod_{i=1}^n\Big( \int_{\R}    g_i(x)^p\dd x\Big)^\frac{1}{p}.
\end{eqnarray*}

This  is equivalently rewritten as
\begin{align*}
\int_{\R^n}\Big( \prod_{i=1}^n f_i(x_i)\Big) \exp\Big\{ &-\frac12\langle \underline x, C^{-1}\underline x\rangle\Big\}\frac{\dd \underline x}{(2\pi)^{\frac{n}{2}}\det(C)^{\frac12}} 
\cr \le &\  \frac{ \big(\prod_{i=1}^n\s_i\big)^{\frac{1}{p}}}{\det(B)^{\frac12(1-\frac{1}{p})} \det(C)^{\frac12}}  \prod_{i=1}^n\Big( \int_{\R}    |f_i(x)|^p e^{-x^2/(2\s_i^2) }\frac{\dd x}{\s_i\sqrt{2\pi}}\Big)^\frac{1}{p},
\end{align*}
namely 
\begin{eqnarray*}
\E\Big( \prod_{i=1}^n f_i(X_i)\Big)\ \le \ 
\frac{ \big(\prod_{i=1}^n\s_i\big)^{\frac{1}{p}}}{\det(B)^{\frac12(1-\frac{1}{p})} \det(C)^{\frac12}}\
 \prod_{i=1}^n\Big( \E   f_i(X_i)^p\Big)^\frac{1}{p}.
\end{eqnarray*}
\vskip 3pt 
Writing 
$$B=C^{-1}-\frac{1}{p}I(\underline{\g}^{-1})=\frac{1}{p}\,C^{-1}I(\underline{\g}^{-1})\Big(p\,I(\underline{\g})- C\Big),$$
 we have
\begin{eqnarray*}
\det(B)=\frac{1}{p^n}\,\det\big(C^{-1}I(\underline{\g}^{-1})\big)\det\Big(p\,I(\underline{\g})-C\Big)=\frac{\det\big(p\,I(\underline{\g})- C\big) }{p^n\,\det(C)\prod_{i=1}^n\s_i^2} .
\end{eqnarray*}
So that,
\begin{eqnarray*}
\E\Big( \prod_{i=1}^n f_i(X_i)\Big)
&\le &
 \frac{ \big(\prod_{i=1}^n\s_i\big)^{\frac{1}{p}}}
 {\big(\frac{\det(pI(\underline{\g})-C)}{p^n\det(C)\prod_{i=1}^n\s_i^2}\big)^{\frac12(1-\frac{1}{p})} \det(C)^{\frac12}}\ \prod_{i=1}^n\Big( \E   f_i(X_i)^p\Big)^\frac{1}{p}
\cr &=&
\frac{p^{\frac{n}{2}(1-\frac{1}{p})} \big(\prod_{i=1}^n\s_i\big)}
 { \det\big(pI(\underline{\g})-C\big)^{\frac12(1-\frac{1}{p})} \det(C)^{\frac{1}{2p}}}\ \prod_{i=1}^n\Big( \E   f_i(X_i)^p\Big)^\frac{1}{p}
.
\end{eqnarray*}

In order to estimate $ \det\big(pI(\underline{\g})-C\big)$, we recall a well-known result on Hadamard matrices.
\begin{lemma}[\cite{O},\,(2)]\label{l2}  Let $A=\{a_{i,j}, 1\le i,j\le n\}$
 and assume that
\begin{eqnarray*}
\sum_{j=1\atop j\neq i}^n|a_{i,j}|\,<\,|a_{i,i}|,\qq \qq i=1,\ldots , n.
\end{eqnarray*}
Then 
\begin{eqnarray*}\det(A)\,\ge \,\prod_{i=1}^n \Big( |a_{i,i}|-\sum_{j=1\atop j\neq i}^n|a_{i,j}|\Big).
\end{eqnarray*}\end{lemma}

By assumption \ref{cond},
\begin{eqnarray*}
\sum_{1\le j\le n\atop j\neq i} |\E X_iX_j|\,\le \,\Big(\frac{p}{2} -1\Big)\, \s_i^2, \qq \qq i=1,\ldots, n.
\end{eqnarray*}
 Letting $pI(\underline{\g})-C= \{d_{i,j},1\le i,j\le n\}$, we have $|d_{i,i}|=(p-1)\s_i^2$ and 
\begin{eqnarray*}
\sum_{j=1\atop j\neq i}^n|d_{i,j}|\,=\, \sum_{j=1\atop j\neq i}^n|\E X_iX_j|\,<\,\Big(\frac{p}{2} -1\Big)\,\s_i^2\,<\, |d_{i,i}|,\qq  i=1,\ldots , n.
\end{eqnarray*}
Thus
$$|d_{i,i}|-\sum_{j=1\atop j\neq i}^n|d_{i,j}| \,\ge\, (p-1)\s_i^2-\Big(\frac{p}{2} -1\Big)\,\s_i^2\,=\, \frac{ p\s_i^2}{2}.$$
By Lemma \ref{l2}, it follows that 
\begin{eqnarray*}
\det\big(pI(\underline{\g})-C\big)\,\ge\, \prod_{i=1}^n\big(p\s_i^2- \sum_{1\le j\le n\atop j\neq i} |\E X_iX_j|\big) \,\ge\, \Big(\frac{ p }{2}\Big)^n\prod_{i=1}^n \s_i^2 .
\end{eqnarray*}
\noi
Thus

\begin{eqnarray*}
\E\Big( \prod_{i=1}^n f_i(X_i)\Big)
 &\le &
\frac{p^{\frac{n}{2}(1-\frac{1}{p})} \big(\prod_{i=1}^n\s_i\big)}
 { \det\big(pI(\underline{\g})-C\big)^{\frac12(1-\frac{1}{p})} \det(C)^{\frac{1}{2p}}}\ \prod_{i=1}^n\Big( \E   f_i(X_i)^p\Big)^\frac{1}{p}
\cr &\le &
\frac{ p^{\frac{n}{2}(1-\frac{1}{p})}\big(\prod_{i=1}^n\s_i\big)^{\frac{1}{p}}}
 {\big(\frac{ p }{2}\big)^n\big)^{\frac12(1-\frac{1}{p})} \det(C)^{\frac{1}{2p}}}\ \prod_{i=1}^n\Big( \E   f_i(X_i)^p\Big)^\frac{1}{p}
\cr &= &
\frac{ 2^{\frac{n}{2}(1-\frac{1}{p})}\big(\prod_{i=1}^n\s_i\big)^{\frac{1}{p}}}
 {\det(C)^{\frac{1}{2p}}}\ \prod_{i=1}^n\Big( \E   f_i(X_i)^p\Big)^\frac{1}{p}
.
\end{eqnarray*}

\end{proof}


\section{\bf Proof of Theorem \ref{p2}.}\label{s3}   
  Recall  Bump and Diaconis \cite[Th.\,4]{BD}  strong Szeg\"o limit theorem. Let $(c_k)_{k\in \Z}$ satisfy conditions \eqref{c1}  and \eqref{c2}. Let $\s(t)= \exp\{\sum_{k\in \Z} c_ke^{ikt}\}$. Then 
\begin{equation}\label{bdsslm}D_{n-1} (\s )\sim \exp\Big( nc_0 + \sum_{k=1}^\infty kc_kc_{-k}\Big).
\end{equation}
    Let $\Gamma_n={\rm Cov}(X_1, \ldots, X_n)$. By applying 
   \eqref{bdsslm} with the choice  $\s(t) = \exp\{\log f(t)\}$, we have
 \begin{equation}\label{sslm}\det (\Gamma_n) \, \sim\, \exp\Big( nc_0+ \sum_{k=1}^\infty kc_kc_{-k}\Big).
\end{equation}

As 
 $$c_0= \frac{1}{2\pi}\int_{- \pi}^{ \pi}
\log f(t) \dd t=G(f),$$    
it follows that
  \begin{eqnarray*} \det (\Gamma_n)
 &\ge & \Big(\frac{1}{1+\d_n}\Big) \exp\Big\{n G(f)+ \sum_{k=1}^\infty kc_kc_{-k}\Big\} 
,\end{eqnarray*} 
where $\d _n\downarrow 0$.  
 Whence 
 \begin{eqnarray*}
\bigg|\E\Big( \prod_{i=1}^n f_i(X_i)\Big)\bigg|
&\le &(1+\d_n)
 2^{\frac{n}{2}(1-\frac{1}{p})}\exp\Big\{ -\frac{1}{2p}\big( nG(f)+ \sum_{k=1}^\infty kc_kc_{-k}\big)\Big\}  \prod_{i=1}^n\Big( \E |  f_i(X_i)|^p\Big)^\frac{1}{p}
\cr
&= &(1+\d_n)\,b(f)^{-\frac{1}{2p}}\,
 2^{\frac{n}{2}(1-\frac{1}{p})} 
 G(f)^{-\frac{n }{2p}}\  \prod_{i=1}^n\Big( \E |  f_i(X_i)|^p\Big)^\frac{1}{p},
\end{eqnarray*}
recalling that $b(f)=\exp\big\{   \sum_{k=1}^\infty kc_kc_{-k}\big\}$
  and $G(f)$ is the geometric mean of $f$. 


\section{\bf Examples.}\label{s4}
  We list  some remarkable classes of examples. 

\subsection{} Let   $X=\{X_j, j\in \Z\}$ be a centered Gaussian stationary  sequence, and assume that $X$ has spectral density. Recall that we have the following representation
 \begin{equation} \label{sgpr} X_k= \sum_{m\in\Z} c_m\xi_{k-m},
 \end{equation}
where $(c_m)\in\ell_2(\Z)$ and $(\xi_j)$ are i.i.d. standard Gaussian. 
Note that 
$$ \E X_kX_\ell =\sum_{m\in\Z} c_mc_{m-k+\ell}. $$
Consider the sections  $X^n=\{X_1,\ldots,X_n\}$, $n\ge 1$.
\subsubsection{}Let  $c_m=|m|^{-1}$, $m\in\Z\backslash \{0\}$, $c_0=0$.  \begin{proposition}\label{gs1}  We have  \begin{eqnarray*} p(X^n) 
   &\le & 4\, (\log n)^2+ \mathcal O (\log n).\end{eqnarray*}
\end{proposition}
The bound of  $p(X^n)$  is in fact optimal, up to some numerical constant.  As $ p(X^n)\ll n$,  this is making   inequality \eqref{ineq} effective. 
 \vskip 2 pt 
We first  prove a lemma. 
\begin{lemma} We have 
\begin{eqnarray*}  \E X_kX_\ell 
 &=& 4\ \frac{\log |k-\ell| }{|k-\ell|}+ \mathcal O \Big(\frac{1}{|k-\ell|}\Big).\end{eqnarray*}
 \end{lemma}
 \begin{proof} We note that  
  $$ \E X_kX_\ell
 =\sum_{m\in\Z\atop m\neq0, m\neq |k-\ell|} \frac{1}{|m||m-|k-\ell||} ,$$
if $k\neq \ell$, and $\E X_k^2= \frac{\pi^2}{3}$.
Let $\m$ be some positive integer. Then,
 \begin{eqnarray*} \sum_{{m\in \Z\atop m\neq\m}\atop m\neq 0} \frac{1}{|m||m-\m|}
&=&  \sum_{m\ge \m +1} \frac{1}{m(m-\m)} + \sum_{m\le \m-1\atop m\neq 0}\frac{1}{|m|(\m-m)}
\cr &=& \sum_{m\ge \m +1} \frac{1}{m(m-\m)} +\sum_{m\le-1}\frac{1}{-m(\m-m)}
+ \sum_{1\le m\le \m-1}\frac{1}{m(\m-m)}
\cr &=& 2\ \sum_{\nu\ge 1} \frac{1}{\nu(\nu+\m)} + \sum_{m=1}^{\m-1}\frac{1}{m(\m-m)}.
\end{eqnarray*}
Recall that 
$\sum_{1\le m\le x} \frac{1}{m}= \log x +\g + \mathcal O (x^{-1})$,
where $\g$ is Euler's constant. 
At first,
\begin{eqnarray*}  \sum_{m=1}^{\m-1}\frac{1}{m(\m-m)}=\frac{1}{\m} 
 \sum_{m=1}^{\m-1}\Big( \frac{1}{m}+\frac{1}{\m-m}\Big)=  \frac{2}{\m} \, 
 \sum_{m=1}^{\m-1} \frac{1}{m}
 =  2\  \frac{\log \m}{\m}+ \mathcal O (\m^{-1}).\end{eqnarray*}
Next
 \begin{eqnarray*}\sum_{\nu\ge 1} \frac{1}{(\nu+\m)\nu}  &=& \sum_{1\le \nu\le \m - 1} \frac{1}{(\nu+\m)\nu} + \sum_{\nu\ge \m} \frac{1}{(\nu+\m)\nu}
 \cr &=& \sum_{1\le \nu\le \m - 1} \frac{1}{(\nu+\m)\nu}+ \mathcal O (\m^{-1})
\cr &=&\frac{1}{\m}\,  \sum_{1\le \nu\le \m - 1} \Big( \frac{1}{\nu}-\frac{1}{\nu+\m}\Big) + \mathcal O (\m^{-1})
\cr &=& \frac{\log \m}{\m} - \frac{1}{\m}\,  \sum_{\m\le h\le 2 \m - 1}\frac{1}{h}+ \mathcal O (\m^{-1})
\cr &=& \frac{\log \m}{\m} + \mathcal O (\m^{-1}). \end{eqnarray*}
Consequently, 
\begin{eqnarray*} \sum_{{m\in \Z\atop m\neq\m}\atop m\neq 0} \frac{1}{|m||m-\m|}
&=& 2\ \sum_{\nu\ge 1} \frac{1}{\nu(\nu+\m)} + \sum_{m=1}^{\m-1}\frac{1}{m(\m-m)}
\cr &=& \frac{4\log \m }{\m}+ \mathcal O (\m^{-1}).\end{eqnarray*}
And so,
\begin{eqnarray*}  \E X_kX_\ell 
 &=& 4\ \frac{\log |k-\ell| }{|k-\ell|}+ \mathcal O \Big(\frac{1}{|k-\ell|}\Big).\end{eqnarray*}
\end{proof}
 \begin{proof}[Proof of Proposition \ref{gs1}]
It follows that 
\begin{eqnarray*}  \sum_{\ell= 1\atop \ell\neq k}^n|\E X_kX_\ell| 
 &=& 4\, \sum_{\ell= 1\atop \ell\neq k}^n\frac{\log |k-\ell| }{|k-\ell|}+ \mathcal O \Big(\sum_{\ell= 1\atop \ell\neq k}^n\frac{1}{|k-\ell|}\Big)
.\end{eqnarray*}
Now,
\begin{eqnarray*}\sum_{\ell= 1\atop \ell\neq k}^n\frac{\log |k-\ell| }{|k-\ell|}&=& 
\sum_{\ell= 1}^{k-1}\frac{\log (k-\ell )}{k-\ell}+\sum_{\ell= k+1}^n\frac{\log (\ell -k )}{\ell-k}
\cr &\le & 
\int_1^{k-1}\frac{\log t}{t}\dd t+\int_1^{n-k}\frac{\log t}{t}\dd t
\cr &=& \frac{\big(\log (k-1)\big)^2 +\big((\log (n-k)\big)^2}{2}.\end{eqnarray*}
Similarly, 
\begin{eqnarray*}\sum_{\ell= 1\atop \ell\neq k}^n\frac{\log |k-\ell| }{|k-\ell|} &\ge & \frac{\big(\log (k-1)\big)^2 +\big((\log (n-k)\big)^2}{2}- (\log 2)^2
.\end{eqnarray*}
Therefore
\begin{eqnarray*}\sum_{\ell= 1\atop \ell\neq k}^n\frac{\log |k-\ell| }{|k-\ell|}&=& 
\frac{\big(\log (k-1)\big)^2 +\big((\log (n-k)\big)^2}{2}+ \mathcal O \big(1\big).\end{eqnarray*}
Also, 
\begin{eqnarray*}\sum_{\ell= 1\atop \ell\neq k}^n\frac{1 }{|k-\ell|}
&\le & 
\int_1^{k-1}\frac{\dd t}{t} +\int_1^{n-k}\frac{\dd t}{t}
\ =\ \log (k-1)+ \log (n-k).\end{eqnarray*}
Whence, \begin{eqnarray*}  \sum_{\ell= 1\atop \ell\neq k}^n|\E X_kX_\ell| 
   &= & 2\,\Big(\big(\log (k-1)\big)^2 +\big((\log (n-k)\big)^2\Big)+ \mathcal O \big(\log k+ \log (n-k)\big).\end{eqnarray*}
   We consequently get the estimate 
 \begin{eqnarray*} p(X) 
   &\le & 4\, (\log n)^2+ \mathcal O (\log n).\end{eqnarray*}
 \end{proof} 
  
\subsubsection{}Now let  $c_m=|m|^{-r}$, $m\in\Z\backslash \{0\}$ where $r\ge 1$, $c_0=0$. Then naturally $X$ has stronger asymptotical independence properties. In fact
$p(X)<\infty$, as soon as $r\ge 2$. 
\vskip 2 pt
 Indeed, by H\"older's inequality, next Proposition \ref{gs1}, 
\begin{eqnarray*}0\ \le \ \E X_kX_\ell
 &=&\sum_{m\in\Z\atop m\neq0, m\neq |k-\ell|} \frac{1}{|m|^r|m-|k-\ell||^r}
\cr &\le &\Big(\sum_{m\in\Z\atop m\neq0, m\neq |k-\ell|} \frac{1}{|m||m-|k-\ell||}\Big)^r
\cr &\le &C_r \,\Big(\frac{\log |k-\ell| }{|k-\ell|}\Big)^r,\end{eqnarray*}
a bound   from which easily follows that $p(X)<\infty$. So, this is an instance where 
Klein, Landau and Shucker's inequality \eqref{dec} directly applies. 

\begin{remark}[A pathological example]   Gaussian stationary  sequences with spectral density form a huge class, and may in particular exhibit pathological covariance functions. We provide here  a simple example of which the study relies on additive Number Theory. 
\vskip 2 pt 
Let $A\subset \N$ and $(b_j)\in\ell^2(\N)$. Consider the following special case of \eqref{sgpr},
 \begin{equation} \label{sgp1} X_k= \sum_{|m|\in A} b_{|m|}\xi_{k-m},\qq k\in \Z.
 \end{equation}
Then 
\begin{equation} \label{ex}\E X_kX_\ell=
\begin{cases}\qq\qq  0\quad & \hbox{if $k-\ell\notin A-A$},\cr 
\displaystyle{\sum_{m\in A\atop
m-k+\ell\in A} b_{|m|}b_{|m-k+\ell|}}\quad & 
  \hbox{otherwise}.
\end{cases}
 \end{equation}
Thus the   covariance function is supported on the difference set $A-A$, making the study of this example depending on additive properties of the set $A$.
\end{remark} 
\vskip 5 pt


\subsection{Hilbert type 
covariance matrices} Now consider non stationary Gaussian sequences   having  Hilbert type 
 matrices.   More precisely, let  $C_n$ be the symmetric matrix defined by
\begin{equation}\label{Polya.Szego}
C_n =\Big\{\frac{1}{a_k+a_\ell}\ ; \ k,\ell=1,\ldots, n\Big\},
\end{equation}
where $A=(a_i)_{i\ge 1}$ is a sequence of positive 
 real numbers. That $C_n$ is positive definite (and so is   a Gram matrix)   follows from the fact that 
$$\int_0^\infty \big| \sum_{k=1}^n x_k e^{-a_k t}\big|^2 \dd t\ge 0.$$
Thus $C_n$ is  the covariance matrix of a Gaussian vector, which can be described explicitly. Indeed, there
 exist in $\R^n$   vectors $u^1,\ldots, u^n$ with Gram matrix $C_n$, for instance the rows of
$C_n^{1/2}$.
Let $\{g_i\}_{1\le i\le n}$ be independent Gaussian standard random variables, and form the Gaussian vector $X^n= \{X_i\}_{1\le i\le n}$ where
$$X_i=\sum_{k=1}^n  g_ku^i_k, \qq i=1,\ldots , n.$$
We immediately see that $X^n$ has covariance matrix $C_n$. Assume that the sequence $A$ is increasing. One easily to check that $p(X^n)\asymp n$. 

\bigskip

\noindent {\bf Acknowledgements.} The author is  grateful to  Abel Klein  for a clarification of a point in Brascamp and Lieb's paper. He also  thank Estelle Basor for friendly and stimulating exchanges around Toeplitz operators and determinants generated
by symbols.



\begin{thebibliography}{99}


\bibitem{B} R. Bellman, (1997) \emph{Introduction to Matrix Analysis}, Second Ed., SIAM, Philadelphia.
\bibitem{BF}  E. L. Basor  and P. J. Forrester, (1994) Formulas for the evaluation of Toeplitz determinants with rational generating functions, \emph{Math. Nachr.} {\bf 170}, 5--18.
\bibitem{BL}  H. J. Brascamp and E. H. Lieb, (1976) Best constants in Young's inequality, its converse, and
its generalization to more than three functions, \emph{Adv. in Math.} {\bf 201}, 151--173.
\bibitem{BD}  D. Bump and P. Diaconis, (2002)  Toeplitz minors, \emph{Journal of Combinatorial Theory} Series A   {\bf  97}, 252--271.
\bibitem{D} K. M. Day, (1975) Toeplitz Matrices Generated by the Laurent Series Expansion of an Arbitrary
Rational Function, \emph{Trans. Amer. Math. Soc.} {\bf  206}, 224--245.
\bibitem{GS} U. Grenander and G. Szeg\"o, (1958)   {\sl Toeplitz forms and their applications}, Univ. of California Press, Berkeley and Los
Angeles.
 \bibitem{KLS}  A. Klein, L. J. Landau and  D. S. Shucker,   (1982)
Decoupling
inequalities for stationary Gaussian processes, \emph{Ann. Probab.}  {\bf 10},
 702--708.
  \bibitem{K}  A. Klein,  \emph{Private communication}.
 \bibitem{O} A. M. Ostrowski, (1952) Note on bounds for determinants with dominant principal diagonal, \emph{Proc. of the A.M.S.} {\bf 3}, No. 1,  26--30.
\bibitem{W} M. Weber, (2013)
 On small deviations of stationary Gaussian processes and related analytic inequalities, \emph{Sankhya A} {\bf 75},  2, 139--170. 

  \end{thebibliography}
 \end{document}